\documentclass[a4paper,12pt]{article}

\usepackage[utf8]{inputenc}
\usepackage{amsthm}
\usepackage{amsfonts}
\usepackage{mathtools}
\usepackage[nottoc]{tocbibind}
\title{On the center of the Hurwitz graph}
\author{Yuval Hovannes Khachatryan}

\theoremstyle{plain} \newtheorem{theorem}{Theorem}[subsection]
\theoremstyle{definition}
\newtheorem*{theorem*}{Theorem}
\theoremstyle{definition} \newtheorem{definition}[theorem]{Definition}
\theoremstyle{definition} \newtheorem{remark}[theorem]{Remark}
\theoremstyle{definition} \newtheorem{lemma}[theorem]{Lemma}
\theoremstyle{definition} 
\theoremstyle{definition} \newtheorem{observation}[theorem]{Observation}
\theoremstyle{definition} 
\theoremstyle{definition} 

\begin{document}

\maketitle
\tableofcontents{}

\section{Introduction}
\subsection{General}
In this work we explore the connection between the Huritz graph and geometric tree graphs and find central elements in the Hurwitz graph.

\section{Preliminaries}
\subsection{Words and the Hurwitz graph}
This work deals with several types of graphs which have a common vertex set,
namely the set of all maximal chains in the non-crossing partition lattice. 

\begin{definition}
Let $G$ be a group, and $S$ a generating set of $G$. Let $g\in G$. 
Then $\left(s_1,\dots  ,s_k \right)\in S^k$ is called \emph{reduced word} of $g$ in terms of $S$ if 
$g = s_1\cdots s_k$ and this is a shortest possible representation of $g$.

\end{definition}

The non-crossing partition lattice was first introduced by Kreweras
\cite{Kreweras}, and may be defined as follows.

\begin{definition}Let $T$ be the set of all transpositions in the symmetric group $S_{n}$ and
let $l(\cdot)$ be the corresponding length function: $l(\pi)$ is
the minimal number of factors in an expression of $\pi$ as a product
of transpositions. Let $c=\left(1,\dots  ,n\right)\in S_n$ be a cycle of length
$n.$ The \emph{non-crossing partition lattice}, $NC(n)$ is the set
$$NC(n) \coloneqq \left\{ \pi\in S_{n}:l(\pi)+l\left(\pi^{-1}c\right)=l\left(c\right)\right\}$$
ordered by
$$\pi\leq\sigma\Leftrightarrow l\left(\pi\right)+l(\pi^{-1}\sigma)=l\left(\sigma\right).$$
\end{definition}

\begin{remark}
Every permutation can be represented as a product of transpositions.
Usually, such a representation is not unique. Having $\pi\leq\sigma$ means
that there exist words $w_{\pi}$ and $w_{\sigma},$ over the alphabet $T$, such that:
\begin{enumerate}
\item $w_{\pi}$ and $w_{\sigma}$ are reduced words in terms of $T$
of $\pi$ and $\sigma$ respectively.
\item $w_{\pi}$ is a prefix of $w_{\sigma}.$
\end{enumerate}
\end{remark}

\begin{definition}
Denote by $F_{n}$ the set of all \emph{maximal chains} in the non-crossing partition
lattice $NC\left(n\right)$.
\end{definition}

It is not difficult to see that there is a bijection between $F_{n}$ 
and the set of all factorizations of $c$ into $l\left(c\right)=n-1$ transpositions.
We shall write each factorization as a word $\left(t_{1},\dots  ,t_{n-1}\right)$ in the alphabet
of transpositions.

Now we proceed to define one of our main objects of study, the Hurwitz graph. In order to do that, we first define two actions on the words 
in $F_n$.
\begin{definition}\label{Hurwitz_Action}
For each $1\leq i\leq n-2$ and $v=\left(t_{1},\dots  ,t_{n-1}\right)\in F_{n}$,
let
\[
R_{i}(v):=\left(t_{1},\dots  ,t_{i-1},t_{i+1}^{t_{i}},t_{i},t_{i+2},\dots  ,t_{n-1}\right)
\]

and
\[
L_{i}(v):=\left(t_{1},\dots  ,t_{i-1},t_{i+1},t_{i}^{t_{i+1}},t_{i+2},\dots  ,t_{n-1}\right).
\]
using the notation $g^{h}:=h^{-1}gh$.
$R_{i}$ and $L_{i}$ are called right and left \emph{Hurwitz
action} on $F_{n}$, respectively. Clearly, $L_{i}=R_{i}^{-1}$.
\end{definition}
Now, we proceed to define the Hurwitz graph.
\begin{definition}\label{HurwitzGraph}
The \emph{Hurwitz graph} $\mathcal{H}\left(S_n\right)$ is the graph
whose vertex set is $F_{n}$ and two vertices are adjacent if
one is obtained from the other by applying a right or left Hurwitz action. 
\end{definition}
Some of the properties of the Hurwitz graph, such as its radius and an upper bound on its diameter, 
were established in \cite{HurwitzGraph}.

We shall use the following map $\Phi$ from $F_n$ to $S_{n-1}$, also defined in \cite{HurwitzGraph}.
\begin{definition}
For $w=\left(t_1,\dots  ,t_{n-1}\right)\in F_n$ define the partial products 
$$\sigma_j \coloneqq  t_j\dots   t_{n-1} \qquad (1\leq j \leq n-1)$$
and the empty product 
$$\sigma_n \coloneqq id.$$
\end{definition}
By definition, $\sigma_j = t_j\sigma_{j+1}$ for every $1 \leq j \leq n-1$. Thus $\sigma_j(i) \neq \sigma_{j+1}(i)$ 
for exactly two values of $1 \leq i \leq n$, and $\sigma_j(i) > \sigma_{j+1}(i)$ for exactly one such value. 
Ignoring $i=n$, the set 
$$A_j \coloneqq \{1 \leq i \leq n-1 | \sigma_j(i) > \sigma_{j+1}(i)\}$$
thus satisfies 
\begin{equation}\label{phieq-1}
\left|A_j\right| \leq 1 \qquad (1 \leq j \leq n-1).
\end{equation}
On the other hand, for $1 \leq i \leq n-1$ for $\sigma_1(i) = c(i) = i+1$ while $\sigma_n(i) = i$ so that 
$i \in \bigcup_{j=1}^{n-1}A_j$. It follows that 
\begin{equation}\label{phieq-2}
\left| \bigcup_{j=1}^{n-1}A_j \right| \geq n-1.
\end{equation}
Combining \eqref{phieq-1} and \eqref{phieq-2} one concludes that 
$$\left| \bigcup_{j=1}^{n-1}A_j \right| = n-1$$
and 
$$\left|A_j\right| = 1 \qquad (1 \leq j \leq n-1).$$
It follows that the map $\pi_w$ defined by 
$$\pi_w(j) \coloneqq i \quad \text{if} \quad A_j = \left\{i\right\}$$
is a permutation in $S_{n-1}$.

We now define a map from $F_n$ to $S_{n-1}$.
\begin{definition}\label{Defn-Phi}
Define $\Phi : F_n \rightarrow S_{n-1}$ by 
$$\Phi(w) \coloneqq \pi_w \qquad \left(\forall w\in F_n \right).$$
\end{definition}

\subsection{Geometric graphs}\label{geometric-graphs}

A \emph{geometric graph} is a graph in which the vertices or the edges are associated with geometric objects.
In this work, all the geometric graphs will have their vertices interpreted as finite sets of points
in the Euclidean plane and the edges as straight lines connecting them. 
.
\begin{definition}
 A \emph{non-crossing geometric tree} of order $n$ is a tree whose vertices form a set of $n$ points in 
 convex position in the plane, and any two edges may intersect only in a common vertex.
\end{definition}
\begin{definition}
For every $n$, the \emph{geometric tree graph} $\mathcal{G}_n$, is the graph
which has as vertices all the geometric trees of order $n$ on a given set of vertices 
(in convex position), and there is an edge
between trees $T$ and $S$ if there exist edges $e$ of $T$ and $f$ of $S$ such that $S = T - e + f$, 
namely $S$ is obtained by removing $e$ from $T$ and adding $f$.
Geometric tree graphs were studied by Hernando, Hurtado, Marquez, Mora and Noy \cite{HHMM}. 
\end{definition}

There is a natural mapping from words in transpositions in $S_n$ to geometric graphs.
\begin{definition} \label{Defn-G}
We define the \emph{geometric tree map} $\Gamma$ from the set of all words in transpositions in $S_n$ 
to the set of all geometric graphs, as following.
For a sequence $w=\left(\left(a_{1},b_{1}\right),\dots  ,\left(a_{k},b_{k}\right)\right)$ of 
transpositions in $S_n$, let $\Gamma\left(w\right)$ be the geometric graph whose vertices are 
$\left\{ 1,\dots  ,n\right\} $ drawn clockwise on a circle, and whose edges are 
$\left\{ \left(a_{i},b_{i}\right):1\leq i\leq k\right\}.$
\end{definition}

\begin{remark}
From now on, for every geometric graph $G = (V, E)$, it is assumed that the vertex set $V$ is a
subset of ${1,\dots,n}$ drawn clockwise on a circle. 

\end{remark}
\begin{remark}\label{edge-notation}
In this work, if a graph and has an edge with endpoints  $a$ and $b$, then it will be denoted as $\left(a \ b\right)$ and not as $ab$ or $\{a, b\}$.
It does not mean that the graph is directed, but rather that the edge corresponds to a transposition $(a\ b)$.
\end{remark}

\begin{definition}\label{BoundaryCatDefinition}
A \emph{boundary caterpillar} $C$ is a non-crossing geometric tree which is a caterpillar and its spine
is either a single vertex or a path of the form $\left(i\ i+1\right), \dots   ,\left( i+k-1 \ i+k \right)$.
\end{definition}

\subsection{Geometric tree order}
\begin{remark}
From now on, $[n]$ will denote the set $\{1,\dots  ,n\}$.
\end{remark}

\begin{definition}\label{cyclic-order}
 Let $k\in \left[n\right]$. We define $<_k$  to be the linear order on $[n]$ defined by 
 $k <_k k+1 <_k\dots<_k n <_k 1 <_k. \dots   <_k k-1$.
\end{definition}

The following theorem, due to Goulden and Yong \cite{GY}, gives a connection between words in $F_n$
and geometric trees.
\begin{theorem}[\cite{GY}, Theorem 2.2]\label{GouldenYongTheorem}
 A word $\left(t_1,\dots  ,t_{n-1}\right)\in T^{n-1}$ (i.e. a sequence of $n-1$ transpositions in $S_n$)
 belongs to $F_n$ if and only if the following
 two properties hold.
\begin{enumerate}
\item $\Gamma(w)$ is a non-crossing geometric tree.
\item $\Gamma(w)$ has cyclically decreasing neighbors: For every $1\leq a \leq n$ and 
$1\leq i < j \leq n-1,$
if $t_i = \left(a, c\right)$ and $t_j = \left(a, b\right)$ then $b<_a c$.
\end{enumerate}
\end{theorem}
\begin{remark}
In the rest of this text, all the geometric trees will be assumed to be non-crossing.
\end{remark}

This theorem allows to define a partial order on the edges of a geometric tree $T$ of order $n$.
\begin{definition}\label{TreePreOrder}
Let $T$ be a (non-crossing) geometric tree. Define a relation $\prec_T$ on the edges of $T$ as follows:
For two adjacent edges $(i\ j)$ and $(i\ k)$, define $\left(i\ j\right) \prec_T \left(i\ k\right)$ if 
$k <_i j$.
\end{definition}
\begin{remark}
The relation $\prec_T$ is antisymmetric.
\end{remark}
\begin{definition} \label{TreeOrderDefinition}
Let $T$ be a (non-crossing) geometric tree. Define the relation $<_T$ on the edges of $T$ as
the transitive closure of $\prec_T$, namely:
for two distinct edges $e$ and $f$, $e <_T f$ if for the unique path $e = e_1, \dots  , e_m = f$, that connects $e$ to $f$ in $T$ (see Remark \ref{Unique-path}) 
we have $e_k \prec_T e_{k+1}$ for every $1 \leq k < m$. 
\end{definition}
The following fact is well known.
\begin{observation}\label{TransitiveClosureProperty}
Let $R$ be an antisymmetric relation on a set $S$ such that for every $x,y\in S$ there is at most 
one finite sequence $x = a_0,\dots, a_n = y$ 
such that $a_{i-1} R a_i$ for every 
$1 \leq i \leq n$.
Then the transitive closure $\bar{R}$ of $R$ is
antisymmetric.
\end{observation}

\begin{theorem}\label{TreeOrderTheorem}
For every geometric tree $T$, $<_T$ is a partial order on the edges of $T$.
\end{theorem}
\begin{proof}
Every finite sequence of edges in $T$, with the property
that every two consecutive edges $e$ and $f$ we have $e \prec_T f$, must form a path. 
Now, between every two edges there is exactly one path, hence at most one sequence as above.
Hence by Lemma \ref{TransitiveClosureProperty}
$<_T$ is antisymmetric.
It is clearly anti-reflexive, hence a strong order on the edges of $T$.
\end{proof}

\begin{remark}\label{ExtensionRemark}
Since edges of $T$ can be viewed as transpositions in $S_n$, a linear extension of $<_T$ defines a word in $F_n$
by Theorem \ref{GouldenYongTheorem}.
Therefore, we obtain another interpretation of $F_n$: 
The elements of $F_n$ are all the linear extensions of the partial orders $<_T$
for all non-crossing geometric trees $T$ of order $n$.
\end{remark}

\section{Linearity and caterpillars}
\subsection{Linearity and caterpillars}
\begin{remark}
From now on, all arithmetic on elements of geometric trees will be done modulo $n$.
\end{remark}
Let $T$ be a geometric tree. Recall Definitions \ref{BoundaryCatDefinition} and \ref{TreeOrderDefinition}. We find  necessary  and sufficient  condition
for $<_T$ to be linear.
\begin{theorem}\label{CtrPlrThm1}
 For a geometric tree $T$, $<_T$ is a linear order  if and only if $T$ is a boundary caterpillar.
\end{theorem}
First we prove several lemmas.
\begin{lemma}\label{path-order-lemma}
If $P$ is a path $\left(j_1\ j_2\right),\dots   ,\left(j_{m-1}\ j_m\right)$ in a geometric tree $T$  such that 
$j_1 <_{j_1} j_2 <_{j_1}\dots   <_{j_1} j_m$, then $<_P$
is a linear order:~
$$\left(j_1\ j_2\right) \prec_P\dots   \prec_P \left(j_{m-1}\ j_m\right).$$ \end{lemma}
\begin{proof}
For every $1 \leq i \leq m-2$ we have $(j_i\ j_{i+1}) \prec_P (j_{i+1}\ j_{i+2})$ by Definition \ref{TreePreOrder}.
\end{proof}

\begin{lemma}\label{star-order-lemma}
If $S$ is a star $(i\ j_1),\dots,(j\ j_m)$ in a geometric graph $T$, such that or every $1\leq k \leq m$,
$j_{k} >_i j_{k+1}$, then $<_S$ is a linear order:
$$(i\ j_1) <_S \dots   <_S (i\ j_m).$$
\end{lemma}
\begin{proof}
Follows from the definition of the geometric tree order. 
\end{proof}

\begin{lemma}\label{CtrPlrEdgeLemma1}
Let $T$ be a non-crossing geometric tree and such that $<_T$ is a linear order. Let $e$ be the $<_T$-smallest edge 
in $T$.
Then for some $1 \leq i \leq n$, $e = (i\ i+1)$ and $i$ is a leaf of $T$. 
\end{lemma}
\begin{proof}
 Assume that $<_T$ is a linear order, and let $e= (i\ j)$ be the $<_T$-first edge. First we show that either 
 $i$ or $j$ must be a leaf if $T$. Assume otherwise.
 Then there are edges $(k\ i)$ and $(j\ m)$ with $k\neq j$ and $m \neq i$ .
The unique path from $(k\ i)$ to $(j\ m)$ is 
 $(k\ i), (i\ j), (j\ m)$ and neither $(k\ i) <_T (i\ j) <_T (j\ m)$ nor $(k\ i) >_T (i\ j) >_T (j\ m)$ holds, 
 because $(i\ j)$ is minimal. Therefore $(k\ i)$ and $(j\ m)$ are not comparable by $<_T$,
 a contradiction to the linearity of $T$.
 
 Without loss of generality, assume that $i$ is a leaf of $T$. Next we show that $j = i+1$.
 Suppose otherwise. Because $T$ is a non-crossing geometric tree and $j$ is the unique neighbor of $i$ there is $k$ between $i$ and $j$, $i <_i k <_i j$
 such that $(k\ j)$ is an edge in $T$. But $i <_i k <_i j$ is equivalent to $j <_j i <_j < k$ which implies that 
 $(j\ k) <_T (i\ j)$, a contradiction to $(i\ j)$ being the minimal edge. Thus we have $j = i + 1$ as claimed.
\end{proof}

Using the same reasoning, we obtain a similar theorem about the maximal elements in $<_T$.
\begin{lemma}\label{CtrPlrEdgeLemma3}
Let $T$ be a non-crossing geometric tree and such that $<_T$ is a linear order. Let $f$ be the $<_T$-largest edge in 
$T$.
Then, for some $1 \leq i \leq n$, $e = (i-1 \ i)$  and $i$ is a leaf of $T$. 
\end{lemma}

\begin{lemma}\label{CtrPlrEdgeLemma2}
Let $T$ be a non-crossing geometric tree such that $<_T$ is a linear order. Then any two $<_T$-consecutive
edges in $T$ are adjacent.
\end{lemma}
\begin{proof}
By definition, $(i\ j) <_T (k\ l)$ if and only if there exists 
 an $<_T$-increasing path from $(i\ j)$ to $(k\ m)$.
 However, since $(i\ j)$ and $(k\ l)$ are $<_T$-consecutive,
 this must be a path of length $2$, which implies that these edges are adjacent. 
\end{proof}

The following definition is helpful to make the argument more clear.
\begin{definition}
Let $S\subseteq [n]$. We say that a graph $G = (V, E)$ is a \emph{cyclic path on S} if  $G$ is a path, $V$
is a cyclic interval $[a, b]_S$ in $S$, and there is an edge between any two $<_a$-consecutive
elements of $[a, b]_S$.
\end{definition}
\begin{observation}\label{cyclic-interval}
A geometric tree $T$ is a boundary caterpillar on 
a subset $S$ of $[n]$, if 
and only if the spine of $T$ is a cyclic path in
 $S$.
\end{observation}
Now we are ready to prove the Theorem \ref{CtrPlrThm1}.

\begin{proof} 
Let $T$ be a boundary caterpillar. We show that any two edges of $T$ are 
$<_T$-comparable. 

Let $e$ and $f$ be edges of $T$. Then there exist vertices $j$ and $k$ in the spine of $T$ 
such that $e$ is an edge of $S_j$ and $f$ is an edge of $S_k$, the stars centered at $j$ and $k$
respectively. Recall that the spine of a boundary caterpillar is of the form $(i\ i+1),\dots  ,(i+m-1\ i+m)$.
If $j=k$ then $e$ and $f$ are in the same star, hence comparable by Lemma \ref{star-order-lemma}.
If $j \neq k$, assume without loss of generality that $j <_i k$. 
Then the path $(j\ j+1)\dots  ,(k-1\ k)$ is linearly ordered by $<_T$ 
and $(j\ j+1) \leq_T (k-1\ k)$ by Lemma \ref{path-order-lemma}
Also, by Lemma \ref{star-order-lemma}, $e \leq_T (j\ j+1)$ and $f \leq_T (k-1 \ k)$, therefore
$e \leq_T f$. Hence $e$ and $f$ are comparable, as desired.

 We prove the converse by induction on $n$, the number of edges.
 For $n = 1,2$, the claim is correct, because for every geometric tree $T$ with 
 one or two edges $<_T$ is a linear order and $T$ is a boundary caterpillar,. Now suppose that $n>2$ 
 and the theorem holds for all $k < n$.
 
 Let $T$ be a non-crossing tree, such that $<_T$ is a linear order. By Lemma \ref{CtrPlrEdgeLemma1} the first edge in $T$ is of the form $(i\ i+1)$ with $i$ a leaf.
 By removing the vertex $i$ and the edge $(i\ i+1)$ we obtain a geometric tree $T'$ on $\{i+1, \dots  ,i-1\}$,
 which is a boundary caterpillar by the induction hypothesis,
 and its minimal edge is adjacent to $(i\ i+1)$ by lemma \ref{CtrPlrEdgeLemma2}. By Lemma 
 \ref{CtrPlrEdgeLemma1}, this edge must be either 
 $(i-1\ i+1)$ with $i-1$ a leaf of $T'$, or $(i+1\ i+2)$ with $i+1$ a leaf of $T'$.
  
 If the minimal edge is $(i-1\ i+1)$, then $i+1$ is not a leaf, and is therefore in the spine of $T'$. 
 $T$ is obtained from $T'$ by connecting a new leaf to a vertex in the spine, hence $T$ is a caterpillar
 with spine $P$  equal to the spine of $T'$, by Observation \ref{CatInduction1}. 
 To show that $T$ is a boundary caterpillar we need to show that $P$ is a cyclic path on $[n]$.
 $T'$ is a boundary caterpillar on $[n] \setminus i$, hence $P$ is a cyclic interval in $[n] \setminus \{i\}$.
 Note that every cyclic path $I$ in $[n] \setminus {i}$ is also cyclic path on $[n]$, unless both $i-1$
 and $i+1$ are vertices of $I$. But $i-1$ is a leaf in $T'$, so it is not a vertex of $P$. Therefore, $P$ 
 is a cyclic path on $[n]$ as required.
 
 If the minimal edge is $(i+1\ i+2)$, then $i+2$ is an end vertex of $P'$, the spine of $T'$, because otherwise
 it would mean that the leaf $i+1$ is a part of $P'$, which is impossible.
 By Observation \ref{CatInduction2}, $T$ is a caterpillar whose spine $P$ is obtained by adding $(i+1 \ i+2)$
 to the edge set of $P'$, and $i+1$ to the vertex set of $P'$. $P'$ is a cyclic path with vertex set $[i+2\ k]$ for
 some $k \in [n] \setminus {i}$. Hence $P$ is a cyclic path on $[n]$, as required.
\end{proof}

\section{Center of the Hurwitz graph}
\subsection{Caterpillars and the Hurwitz graph}
Recall the definitions of the Hurwitz Graph $\mathcal{H}(S_n)$ (Definition \ref{HurwitzGraph}) and of the mapping $\Gamma$ (Definition \ref{Defn-G}). 
In this chapter we need a theorem of Adin and Roichman \cite{HurwitzGraph}.
\begin{theorem}[\cite{HurwitzGraph}, Theorem 10.3]\label{TheoremHurwitzRadius}
The radius of the Hurwitz graph $\mathcal{H}(S_n)$ 
is ${n-1 \choose 2}$.
\end{theorem}
The main result of this section is the following.
\begin{theorem}\label{CenterThm1}
 Every $w\in \mathcal{H}\left(S_n\right)$ such that 
 $\Gamma(w)$ is a boundary caterpillar 
 is a central vertex of 
 $\mathcal{H}\left(S_n\right)$. 
\end{theorem}

First we prove a lemma.
\begin{lemma}\label{caterpillar_first_letter_leaf}
Let $w \in \mathcal{H}\left(S_n\right)$. Suppose 
that $k$ is a leaf in $\Gamma(w)$. If the
transposition
$t$ containing $k$ is the last letter in $w$, then  $t = (k-1\ k)$; and if $t$ 
is the first letter in $w$, then 
$t = (k\ k+1)$.
\end{lemma}
\begin{proof}
Suppose that $k$ is a leaf in $\Gamma(w)$ and the transposition $t$ containing it is the last letter 
in $w$. If $t = (j \ k)$, where 
$j \neq k-1$, then (by connectivity of $\Gamma(w)$ and the assumption that $k$ is a leaf)
there exists $j <_j i <_j k$, such that
$(j\ i)$ is an edge in $\Gamma(w)$, or 
equivalently, a letter in $w$. But then, by Theorem 
\ref{GouldenYongTheorem}(2),
$(j\ k)$ appears in $w$ before $(i\ j)$, a 
contradiction to $(j\ k)$ being the 
last letter in $w$. The case of $t$ being the first letter is similar.
\end{proof}
\begin{observation}
Let $w = (t_1\dots   t_{n-1}) \in \mathcal{H}\left(S_n\right)$ such that $\Gamma(w)$ is a boundary caterpillar,
$t_k = (i\ j)$ for some $1 \leq k \leq n-1$  and $i$ is a leaf in $\Gamma(w)$.
Then for $w' = (t_1,\dots  , t_{k-1},t_{k+1},\dots,
 t_{n-1})$, $\Gamma(w')$ is a boundary caterpillar on 
$\{1,\dots  , n\} \setminus \{i\}$.
\end{observation}
The following definition is useful.
\begin{definition}
Let $w\in \mathcal{H}\left(S_n\right)$ and suppose that the $k$-th letter in $w$ 
is $t_k = (i\ j)$. Then {\em pushing $t_k$ to $l$-th place preserving $i$} means the following.
If $l < k$ we apply $H_{k-1}, \dots   H_{l}$ in succession where 
$H_m = R_m$ if $t_m$ fixes $i$ and $H_m = L_m$ otherwise. 
If $k < l$ we perform the following. We apply $H_k, \dots   H_{l-1}$ in succession, where
$H_m = L_m$ if $t_{m+1}$ fixes $i$ and $H_m = R_m$ 
otherwise.
Note that by applying an $i$-preserving push from $k$ to $l$ we obtain a new word $w'$ whose 
$l$-th letter doesn't fix $i$, and for every $m$ between $k$ and $l$ (including 
$m$ and excluding $l$), $t_m$ fixes $i$. This push costs $|k-l|$ Hurwitz actions.
\end{definition}

We now prove Theorem \ref{CenterThm1}.
\begin{proof}
By induction on $n$. Clearly, for  $n = 1, 2$ the statement is correct.
Assume that the statement holds for $n-1$. Let $w=\left(t_1,\dots,t_{n-1}\right) \in \mathcal{H}\left(S_n\right)$ such that
$\Gamma(w)$ is a boundary caterpillar and let $u\in \mathcal{H}\left(S_n\right)$. By Theorem \ref{CtrPlrThm1} and Lemma \ref{CtrPlrEdgeLemma1}, $t_1 = (i\ i+1)$ such that 
$i$ is a leaf in $\Gamma(w)$. We pick the last 
transposition $s_k$ in $u$  such that $s_k$ moves 
$i$. 
Now we push $s_k$ to the first place in $u$, 
preserving $i$,
and obtain a new word $u'$ such that $i$ is a leaf in $\Gamma(u')$ and the transposition $s'_1$
containing $i$ is the first letter in $u'$. By Lemma \ref{caterpillar_first_letter_leaf},
$s'_1=(i\ i+1)=t_1$. Now we apply the induction hypothesis. $\left(t_2,\dots  ,t_{n-1}\right)$ and 
$\left(s'_2,\dots,   s'_{n-1}\right)$ are both words in $F_{n-1}$(Actually, the product is not the 
cycle $(1,\dots,n)$ but the $n-1$-cycle 
$(i+1,\dots,n,1,\dots,i-1)$ but they give isomorphic
Hurwitz and geometric tree graphs.)
The word $(t_2,\dots,   t_{n-1})$ is linearly ordered and thus central, by the induction hypothesis. 
Therefore, by Theorem \ref{TheoremHurwitzRadius} $dist\left(\left(t_2,\dots,   t_{n-1}\right), \left(s'_2,\dots,   s'_{n-1}\right)\right) \leq {n - 2 \choose 2}$. Therefore the number of Hurwitz actions needed to get $w$ from $u$ is at most 
$n-2 + {n-2 \choose 2} = {n-1 \choose 2}$, which is the radius of $\mathcal{H}\left(S_n\right)$. Therefore for every
$u\in \mathcal{H}\left(S_n\right)$, $dist(u, w)$ is at most the radius of $\mathcal{H}\left(S_n\right)$, which means that $w$ is 
a central vertex of $\mathcal{H}\left(S_n\right)$.
\end{proof}

\subsection{Central vertices in the Hurwitz graph}
The result of the previous section can be generalized.
We prove the following theorem.
\begin{theorem}\label{CenterThm2}
For every geometric tree $T$ on $n$ vertices there is $w\in \mathcal{H}\left(S_n\right)$ such that $\Gamma(w)=T$ 
and $w$ is central in $\mathcal{H}\left(S_n\right)$.
\end{theorem}

To prove this result, we need the following facts about geometric trees. 
\begin{lemma}\label{LemmaConsequentLeaf}
For every geometric tree $T$ there is a leaf $i$
such that either $(i-1\ i)$ or $(i\ i+1)$ is an edge
of $T$.
\end{lemma}
\begin{proof}
Pick a leaf $i$ in $T$ with neighbor $j$, such that the distance $d=|i-j|$ is minimal. If $d = 1$, we are done.
If not there is a leaf $k\neq j$ in the tree $S$ spanned by $j$ and the vertices strictly between $j$ and $i$ which is also a leaf of $T$. Clearly the distance of $k$ from its neighbor is less than $d$, a contradiction to $d$ being the minimal distance of leaf from its neighbor in $T$.
\end{proof}
The following fact is immediate from Definition \ref{TreeOrderDefinition}.
\begin{observation}\label{LemmaMinimalMaximalLeaf}
Let $T$ be a geometric tree and $i$ a leaf of $T$. If $t=(i\ i+1)$  is an edge of $T$, then $t$ is $<_T$-minimal and if $t= (i-1\ i)$ is an edge of $T$, then $t$ is $<_T$-maximal.
\end{observation}

We now prove Theorem \ref{CenterThm2}. 
\begin{proof}
Again, we prove the theorem by induction on $n$, the number of vertices.
Clearly, the statement is true for $n = 1, 2, 3$. Assume that the statement is correct
for every $m<n$.

Let $T$ be a geometric tree with $n$ edges.
By Lemma \ref{LemmaConsequentLeaf}
there is a leaf $i$ such that either $(i\ i+1)$ or $(i\ i-1)$ is an edge in $T$. 
Without loss of generality, assume that $(i\ i+1)$ is an edge in $T$. The case of $(i\ i-1)$ will be dealt with later.

By Observation \ref{LemmaMinimalMaximalLeaf}, $(i\ i+1)$ is $<_T$-minimal. 
We proceed to build the word $w$ as follows. Set $(i\ i+1)$ to be the first letter in $w$.
By removing $i$ from the vertex set of $T$ and $(i\ i+1)$ from the edge set of $T$ we a obtain a 
non-crossing geometric tree $T'$ on $[n] \setminus \{i\}$.
By the induction hypothesis, there is a central vertex $w'$ of $\mathcal{H}\left(S_{[n] \setminus \{i\} }\right)$
such that $\Gamma(w')=S$.

We prove that $w = \left( \left(i\ i+1\right), w'\right)$ is the required central word in $\mathcal{H}(S_n)$.
For every $u\in F_n$, we pick the last letter $t$
of $u$ containing $i$ and push it to the first place, preserving $i$.
This action requires at most $n-2$ steps and we obtain a new word $u'=(s'_1,\dots,s'_{n-1})$, such that
$\Gamma\left(\left(s'_2,\dots,s'_{n-1}\right)\right)$ is a geometric tree on $[n] \setminus \{i\}$, by Lemma \ref{CtrPlrEdgeLemma1}, $s'_1=(i\ i+1)$.
Now $w'$ is central, hence the distance of $w'$ from $(s'_2,\dots,s'_{n-1})$ is at most ${n-2 \choose 2}$.
Therefore the distance of $u$ from $w$ is at most $n-2 + {n-2 \choose 2}={n-1 \choose 2 }$. 

If there is no leaf $i$ such that $(i\ i+1)$ is edge in $T$, we repeat the process with $(i-1\ i)$, but this time we build a word $w$ such that $(i-1\ i)$ is the last letter
of $w$.
\end{proof}

\newpage{}
\section{Extension of the Hurwitz graph}
\subsection{Basic definitions}
The Hurwitz graph that was studied in the previous section can be extended.
The motivation for the definition is to look at $F_n$ and have an edge between two words 
if they have a common sub-word of length $n-2$.
Formally, the \emph{extended Hurwitz action} and the \emph{extended Hurwitz graph} are defined in a way similar to 
Definitions \ref{Hurwitz_Action} and \ref{HurwitzGraph}.
\begin{definition}\label{ExtendedHurwitzAction}
The {\em extended Hurwitz actions} for $i<j$ are defined as follows:
\begin{enumerate}
\item Extended left Hurwitz action: $L_{ij}\left(w\right):=L_{j-1}L_{j-2}\dots   L_{i}\left(w\right)$
\item Extended right Hurwitz action: $R_{ij}\left(w\right) :=R_{i}R_{i+1}\dots   R_{j-1}\left(w\right)$
\end{enumerate}
\end{definition}
Recall that $t^s = s^{-1}ts$, namely conjugating $t$ by the inverse of $s$.
\begin{remark}
The extended Hurwitz actions can be understood as follows. If $w=(t_1,\dots   ,t_{n-1})$ is a word in $F_n$, 
then 
$$L_{ij}(w) = (t_1,\dots  , t_{i+1},\dots  , t_j,t_i^{t_{i+1}\cdots t_j} ,\dots   ,t_{n-1}).$$
Namely, if the sub-word of $w$ that 
begins with the $i$-th letter and ends with the $j$-th letter is $(t_i, U)$, then the corresponding sub-word of $L_{ij}(w)$  is $(U, t_i^{U})$. The sub-word $U$ is moved to the left, and the letter $t_i$, conjugated 
by $U^{-1}$, is moved to the right.

$R_{ij}$ works in a similar way:
$$R_{ij}(w) = (t_1,\dots,t_j^{t_{j-1}\cdots t_i},t_i,\dots  ,t_{j-1},\dots   ,t_{n-1}). $$
Namely, if the sub-word of $w$ that 
begins with the $i$-th letter and ends with the 
$j$-th letter is $(U, t_j)$, then the corresponding sub-word of $L_{ij}(w)$ 
is $(t_j^{U^{-1}}, U)$. The sub-word $U$ is moved to right, 
and the letter $t_j$ conjugated by the sub-word $U$ is moved to the left.
\end{remark}
\begin{remark}
When we apply $L_{ij}$ to a word, we say that we \emph{pull} the $i$-th letter \emph{to the right}. We
also say that $L_{ij}(w)$ is obtained from $w$ by pulling the $i$-th letter of $w$ to the $j$-th place.
When we apply $R_{ij}$ to a word, we say that we \emph{pull} the $j$-th letter \emph{to the left}.
We also say that $R_{ij}(w)$ is obtained from $w$ by pulling the $j$-th letter of $w$ to the $i$-th place.
\begin{observation}
An extended Hurwitz action, applied to a word $w$,
alters the graph $\Gamma(w)$ in at most one edge.
\end{observation}
\begin{observation}
For every $w\in F_n$ and $1\leq i < j \leq n-1$, the
following relation holds: 
$$R_{ij}L_{ij}\left(w\right)=L_{ij}R_{ij}\left(w\right)=w$$
\end{observation}
\end{remark}
\begin{definition} \label{ExtendedHurwitzGraph}
The {\em extended Hurwitz graph},
denoted by $\mathcal{E}(S_n)$, is defined as follows. The vertex set
is $F_{n}$.
Two words $w_{1}$ and $w_{2}$ are adjacent if there exist $1\leq i<j\leq n-1$
such that either $w_{1}=R_{ij}(w_{2})$ or $w_{1}=L_{ij}(w_{2})$.
\end{definition}
\begin{remark}
The extended Hurwitz graph is undirected.
\end{remark}

\subsection{Radius and diameter of $\mathcal{E}(S_n)$}
In this section we calculate the radius of the extended Hurwitz graph, 
and give bounds on its diameter. See Section \ref{GraphPrelims} for 
the definitions.
\begin{definition}
We say that $w \in F_n$ is a \emph{star} if $\Gamma(w)$ is a star. $\mathcal{S}_i$ denotes the word $w \in F_n$
such that $\Gamma(w)$ is the star centered at $i$.
\end{definition}
\begin{theorem}\label{Extended-Radius}
 The radius of the extended Hurwitz graph $ \mathcal{E}(S_n)$ is $n-2$, and every star is a central 
 vertex in $ \mathcal{E}(S_n)$.
\end{theorem}
To prove the theorem, we need the following lemma.
\begin{lemma}\label{ExtendedpullLemma1}
Let $w\in F_n$ such that $(1\ i)$ is a letter in $w$ and $i$ is not a leaf in $\Gamma(w)$. Then there exists
$j$ such that $(i\ j)$ is a letter in $w$, and pulling $(i\ j)$ results in a word $w'$ that has, instead of $(i\ j),$ a letter $(1\ k)$
which is not a letter in $w$. 
\end{lemma}
\begin{proof}
Since $i$ is not a leaf in $\Gamma(w)$, there is $j\neq 1$ such that $(i\ j)$ is an edge in $\Gamma(w)$, 
or equivalently, a letter in $w$. We consider the case of $i < j$ first.

We can always assume that $j$ is the maximal element of $[n]$ such that $(i\ j)$ is a letter in $w$. Otherwise we replace it with the maximal.
 
We claim that pulling $(i\ j)$ to the left,  to the position
preceding $(1\ i)$, results in a word $w'$ with a letter $(1\ k)$ which is not a letter in $w$. 
Note that by assumptions $(1\ i) <_{\Gamma(w)}$ $(i\ j)$.

We prove this by induction on $j-i$. If $j-i = 1$
 (or more generally: if there is no $h$ such that 
$i < h < j$ and $(j\ h)$ is 
a letter in $w$), we show that $(i\ j)$ commutes with with every $(k\ l)$ such that  
$(1\ i) <_{\Gamma(w)} (k\ l) <_{\Gamma(w)} (i\ j)$ (which means that $(k\ l)$ is a letter between
$(1\ i)$ and $(i\ j)$ in $w$).

Assume otherwise. Then either $l = i$ or $l = j$.
If $l = i$ then $(1\ i) <_{\Gamma(w)} (i\ k) <_{\Gamma(w)} (i\ j)$ implies $j <_i k <_i 1$, equivalently $j < k \leq n$, a contradiction to $j$ being the maximal neighbor of $i$ in $\Gamma(w)$.

If $l = j$ then $(j\ k) <_{\Gamma(w)} (i\ j)$ 
implies
$i <_j k$, equivalently $i < k < j$, a contradiction 
to the assumption that there are no such neighbors 
of $j$.
Hence every letter between
$(1\ i)$ and $(i\ j)$ commutes with $(i
 j)$. Thus if we pull $(i\ j)$ to the position 
 preceding $(1\ i)$, $(i\ j)$ turns into  $(1\ j)$ which
is not a letter in $w$ (since $(1\ i)$ and $(i\ j)$
are letters, and $\Gamma(w)$ is a tree).

Suppose that the statement holds whenever $j-i<m$. Then we show that it holds when $j-i=m$. Let $l$ be minimal neighbor of $j$ in the  interval $[i+1\dots   j-1]$. (If there is no such $l$ the proof is as in case  $j-i=1$.) 
We claim that $(j\ l)$ is the largest letter
between $(1\ i)$ and $(i\ j)$ that doesn't commute with $(i\ j)$. 
 
The only letters between $(1\ i)$ and $(i\ j)$ that do not
commute with $(i\ j)$ are the ones that involve $j$. (The case of $i$ is excluded as in the proof of case $j=i+1$.)
Again, if $i < l < h < j$ and $h$ is a neighbor of $j$, then $(j\ h) <_{\Gamma(w)} (j\ l) <_{\Gamma(w)}(i\ j)$ contradicting the choice of $l$. 

If $(j\ l)<(i 1)$, then there are no letters between $(1\ i)$ and $(i\ j)$, and we continue as in the case of $j = i+1$. Otherwise every letter between $(j\ l)$ and $(i\ j)$
commutes with $(i\ j)$. Pulling $(i\ j)$ immediately before $(j\ l)$ we obtain a new word $w'$ with $(i\ l)$ before $(j\ l)$. 
But $0<l-i < j-i$, hence we can apply the induction hypothesis and pull $(i\ l)$ immediately before $(1\ i)$ to obtain $(1\ k)$ immediately before $(1\ i)$.
But pulling $(i\ j)$ before $(j\ l)$ to obtain $(i \ l)$ and then pulling $(i\ l)$ before $(1\ i)$ to obtain $(1\ k)$ is same as pulling $(i\ j)$ before 
$(1\ i)$. 

Finally, if $i > j$, we apply the same  argument with opposite direction of inequalities.
\end{proof}

\begin{remark}
The lemma remains valid if we replace $1$ by any $i$. The proof is exactly the same, except that $<$ is replaced by $<_i$.
\end{remark}
Now we prove Theorem \ref{Extended-Radius}.
\begin{proof}
 First we prove the theorem for the star $\mathcal{S}_1 = (1\ n)\dots   (1\ 2)$. Let $w\in  \mathcal{E}(S_n)$. 
 We show that the distance of $w$ from $\mathcal{S}_1$ is equal to $d$, the number 
 of letters in $w$ that are not in $\mathcal{S}_1$. It is clear that the distance is at 
 least $d$, because every 
 extended Hurwitz action changes at most one edge in $\Gamma(w)$. 
 Now we shall show that the distance is at most $d$. 
 
 We proceed by induction on $d$. For $d = 1$, it means that the word $w$ has $n - 2$ common 
 letters with $\mathcal{S}_1$. Thus it must 
 have a unique letter $(i\ j)$ with $i, j \neq 1$. It means that either $(1\ i)$ or $(1\ j)$ (but not both) is a letter 
 in $w$. 
 Without loss of generality suppose that $(1\ i)$ is a letter in $w$. Now we apply Lemma
 \ref{ExtendedpullLemma1}, and obtain a new word $w'$ with $(i\ j)$ replaced by the 
 a $(1\ k)$ that is not in $w$ (namely $k=j$). The other letters in $w$ are unaffected by the pull, and all of them are also 
 letters of $\mathcal{S}_1$, hence all the letters in the new word 
 and $\mathcal{S}_1$ are equal. Since $<_{\Gamma(\mathcal{S}_1)}$ is a linear order by Lemma \ref{star-order-lemma},
 the words are equal.
 
 Now suppose that the claim holds for all words 
 which differ from $S_1$ in $m < d$ letters. Let $w$ be a word such that 
 $w$ and 
 $\mathcal{S}_1$ differ in exactly
 $d$ letters. The word $w$ must have a letter of the form $(1\ i)$ such that $i$ is not a leaf in $\Gamma(w)$. 
 We again apply Lemma \ref{ExtendedpullLemma1} to obtain a new word $w'$ with letter $(1\ k)$ that is not 
 a letter in $w$. The word $w'$ therefore has $d-1$ letters that  are not in $\mathcal{S}_1 \in \mathcal{E}\left(S_n\right)$, hence (by the induction hypothesis) the distance of $w'$ from
 $\mathcal{S}_1$ is $d-1,$ and so the distance between $w$ and $\mathcal{S}_1$ is at most $d-1+1 = d$.
 
 Note that for every $w \in F_n$, $w$ and $\mathcal{S}_1$ have at least one common letter, hence the distance
 between $w$ and $\mathcal{S}_1$ is at most $n-2$. Therefore the radius of $\mathcal{E}\left(S_n\right)$ is at most $n-2$.
 
 Also, for every $w \in F_n$ there is $u \in F_n$ such that 
 $u$ and $w$ have at least $n-2$ different letters,
 for example $S_i$ for every leaf $i$ in $w$. Hence the radius of $\mathcal{E}\left(S_n\right)$ as at least $n-2$.
 To prove that every star in $\mathcal{E}\left(S_n\right)$ is central, we apply the same arguments as before to $\mathcal{S}_i$,
 replacing $1$ by $i$ and $<$ by $<_i$. 
\end{proof}
 
Next, we find bounds for the diameter of the extended Hurwitz graph. To do that 
we need a theorem by Hernando, Hurtado, Marquez, Mora and Noy \cite{HHMM} and two lemmas
that establish a connection between $\mathcal{E}(S_n)$ and its image in $\mathcal{G}_n$.

\begin{theorem}[\cite{HHMM}, Theorem 3.5]\label{DiameterOfGeometricTreeGraph}
The diameter of $\mathcal{G}_n$ is at least $\frac{3}{2}n - 5$.
\end{theorem}
\begin{lemma}\label{ImageOfExtendedInGnNeighbors}
For any two words $w, v \in F_n$, if $w, v$ are neighbors in $\mathcal{E}(S_n)$ then $\Gamma(w)$ and $\Gamma(v)$ 
are either neighbors or equal in $\mathcal{G}_n$.
\end{lemma}
\begin{proof}
Every extended Hurwitz action alters at most one letter, which means that $\Gamma(w)$ and $\Gamma(v)$ have at most 
one different edge, which means that they are either neighbors or equal in $\mathcal{G}_n$. 
\end{proof}
\begin{lemma}\label{ImageOfExtendedIsOnto}
The mapping $\Gamma: F_n \rightarrow \mathcal{G}_n$ is onto.
\end{lemma}
\begin{proof}
Let  $T$ be a non-crossing geometric tree. By Theorem \ref{TreeOrderTheorem}, $<_T$ is always a partial order on the edges of $T$. Any linear extension of $<_T$ can be viewed as a word $w$ satisfying
the conditions of Theorem \ref{GouldenYongTheorem}, hence is a word in $F_n$. Obviously, $\Gamma(w) = T$.
\end{proof}

\begin{theorem}\label{Extended-Diameter}
The diameter of the extended Hurwitz graph $\mathcal{E}(S_n)$ is at least $\frac{3}{2}n-5$ and at most $2n - 4$.
\end{theorem}
\begin{proof}
The upper bound $2n - 4$ is trivial, being twice the radius of $\mathcal{E}(S_n)$.
Now we show that the diameter of $\mathcal{E}\left(S_n\right)$ is at least the diameter of $\mathcal{G}_n$. 
Let $T$ and $S$ be geometric trees such that the distance between them in $\mathcal{G}_n$ is equal to the diameter of $\mathcal{G}_n$.
By Lemma \ref{ImageOfExtendedIsOnto} there are vertices $w$ and $u$ of $\mathcal{E}\left(S_n\right)$ such that $\Gamma(w) = T$ and $\Gamma(u) = S$. Let $P$ be a
shortest path from $w$ to $u$ in $\mathcal{E}\left(S_n\right)$. By Lemma \ref{ImageOfExtendedInGnNeighbors}, $\Gamma(P)$ is a walk from $T$ to $S$ in $\mathcal{G}_n$. Therefore the length of $P$ is greater or
equal to the distance between $T$ and $S$, which is equal to the diameter of $\mathcal{G}_n$. 
Hence we have the following inequality:
$$d(w,v) = length(P) \geq length\Gamma(P) \geq
d(T,S) = diam(\mathcal{G}_n) \geq \frac{3}{2}n - 5$$
which proves our claim.
\end{proof}

\newpage{}

\end{document}